\RequirePackage{fix-cm}
\documentclass[12pt,reqno]{amsart}
\usepackage{blindtext}
\AtBeginDocument{%
  \setlength{\oddsidemargin}{\dimexpr(\paperwidth-\textwidth)/2-1in}%
  \setlength{\evensidemargin}{\oddsidemargin}%
  \setlength{\topmargin}{%
    \dimexpr(\paperheight-\textheight)/2-\headheight-\headsep-1in}%
}

\usepackage{graphicx}

\usepackage{rotating}
\usepackage{tikz}
\usepackage{crop}
\usepackage{graphicx}
\usepackage{array}
\usepackage{color}

\usepackage{amsmath,amssymb}

\usepackage{algorithm}
\usepackage{graphicx}
\usepackage{bm}

\usepackage[noend]{algpseudocode}

\makeatletter
\def\BState{\State\hskip-\ALG@thistlm}
\makeatother

\usepackage{flushend}
\usepackage{stfloats}
\usepackage{times}
\usepackage{tabularx}
\usepackage{pstricks}

\newtheorem{Thm}{Theorem}[section]
\newtheorem{definition}[Thm]{Definition}
\newtheorem{lemma}[Thm]{Lemma}
\newtheorem{proposition}[Thm]{Proposition}
\newtheorem{theorem}[Thm]{Theorem}
\newtheorem{example}[Thm]{Example}

\newcommand{\bigxi}{{{\Xi}}}
\newcommand{\diffxi}{\nabla_{{\Xi}}}

\newcommand{\realn}{\mathbb{R}}
\newcommand{\zahlen}{\mathbb{Z}}
\newcommand{\greens}{G}

\newcommand{\braket}[1]{[\![ {#1} ]\!]} 
\newcommand{\braketsgn}[1]{{[\![ {#1} ]\!]}_{{sgn}}}
\newcommand{\brakett}[1]{{[\![ {#1} ]\!]}_{+}}

\newcommand{\fequiv}{\longleftrightarrow}
\newcommand{\sgn}{ {sgn}}

\newcommand{\bpsquash}[1]{\bm{\mu_{{ #1}}}}
\newcommand{\bidxsquash}[1]{\bm{\sigma_{{ #1}}}}

\newcommand{\znorm}[1]{\left|\left| #1 \right|\right|_0 }

\begin{document}

\title{Fast and exact evaluation of box splines via the PP-form}



\author[J. Horacsek]{Joshua Horacsek}
\address{Department of Computer Science, University of Calgary, 2500 University Dr NW, Calgary, AB T2N 1N4}
\email{joshua.horacsek@ucalgary.ca}

\author[U. Alim]{Usman Alim}
\address{Department of Computer Science, University of Calgary, 2500 University Dr NW, Calgary, AB T2N 1N4}
\email{joshua.horacsek@ucalgary.ca}




\begin{abstract}
For the class of non-degenerate box splines, we prove that these box splines are piecewise polynomial. This is not a new result, it is in fact a well known and useful property of box splines. However, our proof is constructive, and the main result of this work is a corollary that follows from this proof, namely one that gives an explicit construction scheme for the polynomial pieces in the interior regions of any non-degenerate box spline.

\end{abstract}

\maketitle

\section{Introduction} \label{sec:intro}
Box splines are a multivariate generalization of B-splines \cite{deboorbox} that are particularly useful in visualization and signal processing applications. 
They allow practitioners to tailor approximation schemes (on regular grids) to specific computational and theoretical parameters.
These parameters include {\em support size} (how many samples are needed to reconstruct a value), {\em smoothness} of reconstruction, and {\em order of approximation}.
There is particular interest in using box splines to span approximation spaces on non-Cartesian lattices, since certain non-Cartesian lattices, like the Body-Centered-Cubic (BCC) lattice, allow more efficient sampling and reconstruction schemes \cite{kimastar}.

Typically, to evaluate a box spline, one either relies on the recursive scheme \cite{deboorbox,boxeval,boxnumer}, or one must derive an explicit representation of the spline -- which has so far been done on a case-by-case basis (see  \cite{pracbox} as an example). 
In this paper, we prove that a box spline's piecewise polynomial from always exists by explicit construction. 
We give a set construction recipe that yields the explicit truncated piecewise polynomial form for any box spline.
From this we derive the explicit polynomial within each region of evaluation.
Once these have been computed, we construct a binary space partitioning tree that splits the support of the box spline into regions of evaluation. 
Each leaf of the tree then gets an explicit polynomial. 
To evaluate the box spline, one must simply traverse the tree, then evaluate that polynomial.

When one wishes to evaluate a box spline, the task is not-so trivial.
The most widely known evaluation scheme is probably deBoor's recursive formula \cite{deboorbox}, which works well in low dimensions, but suffers from numerical issues in higher dimensions \cite{boxnumer,kimeval}. 
Kobbelt analysed and addressed these issues \cite{boxeval}, yet there still appears to be numerical issues in high dimensions \cite{kimeval}.
Moreover, these evaluation schemes tend to be fairly slow, since their evaluation schemes are recursive, and rely on solving a system of equations at each step of the recursion.

In the work of Kim et al \cite{kimeval}, a fast and stable scheme is derived by creating an indexing system that can quickly identify each region of evaluation.
Then within each region of evaluation, the Bernstein B\'{e}zier (BB) form of the piecewise polynomial is derived from a stable form of the recursive evaluation formula.
However, there are a few pitfalls to their method. 
Firstly, it relies on a stable evaluation scheme to derive the BB form of each polynomial.
Secondly, its indexing scheme must be derived on a case-by-case basis.
Finally, it relies on certain assumptions about the direction vectors that constitute a box spline, namely that they be rational.

The only assumptions we make in this work are that our direction vectors are real valued, and must span the target space.
Additionally, the method proposed in this work is exact, our decomposition makes no approximations, and only relies on operations that can be performed exactly within a machine. We do this by incrementally splitting the Fourier representations of a box spline into two parts, a difference operator and a Green's function, then deriving their spatial forms by a series of set constructions. 
The final result comes from a semi-discrete convolution between these two representations.
The paper proceeds as follows. In Section \ref{sec:bk} we provide a basic introduction to box splines and the notation we use to proceed with this work. 
In Section \ref{sec:support}, we derive the spatial form for the difference operator via a simple set construction akin (distributionally equivalent) to the inverse discrete Fourier transform of the difference operator.
We proceed in Section \ref{sec:greens} to construct the spatial form of the Green's function. At first glance, the Fourier form of the Green's function appears to be non-separable, which complicates finding a spatial representation.
However, by constructing vectors from the kernel of the matrix that defines a box spline, we can always construct a distributionally equivalent form that is separable, and  therefore has an easy to characterize inverse Fourier transform. 
We then combine those forms via a semi-discrete convolution.
In Section \ref{sec:stable_eval} we provide our main result, and use it to derive a fast and stable algorithm for exactly evaluating box splines.
Finally, in Section \ref{sec:examples}, we provide example constructions for some known box splines.

\section{Background} \label{sec:bk}
Following the established notation \cite{deboorbox}, we define a box spline as a generalized function $M_\bigxi : \realn^s \rightarrow \realn$ which is characterized by a list of $n$ column vectors ${\bm{\xi}}_1, \ {\bm{\xi}}_2, \ \cdots \ {\bm{\xi}}_n $, with each $\bm{\xi}_i \in \realn^s $, that are collected in an $s \times n$ matrix $ \bigxi := \left[ {\bm{\xi}}_1 \ {\bm{\xi}}_2 \ \cdots \ {\bm{\xi}}_n \right].$

The function $M_\bigxi$ is defined recursively by the convolution equation 
\begin{equation}
	M_{[ \ ]} := \delta 
	\text{   and   }
	M_{\bigxi} := \int^1_0 M_{\bigxi \backslash \bm\zeta}(\cdot -t\bm\zeta)dt 
\end{equation} 
where $\delta$ is the Dirac delta distribution. 
We say the box spline is non-degenerate if the dimension of the range of $\bigxi$ is $s$. 
When we speak of the kernel of a matrix, we speak of a specific basis for the null-space of that matrix. 
We explicitly choose the basis for the kernel so that it is in column echelon form. 
We use the notation $\ker{X}$ as shorthand for the null-space of $X$. For example, $\ker\Xi$ is the  $n \times (n-s)$ matrix that forms a basis for the null-space of $\bigxi$ and is in column echelon form. 
We write $\bm{z} \in \ker{X}$ if $X\bm{z}=\bm{0}$.

If a hyperplane defined by columns of $\Xi$ forms an $s-1$ dimensional subspace of $\realn^s$, then we collect it in the set $\mathbb{H}(\Xi)$. 
We define the mesh of $\Xi$ as 
\begin{equation}
	\Gamma(\Xi) := \bigcup_{H\in\mathbb{H}} H + \Xi\mathbb{Z}^n.
\end{equation}
This is slightly different than the conventional definition of $\Gamma(\Xi)$, projecting $\mathbb{Z}^n$ along $\Xi$ allows us to consider real valued direction vectors in our construction. We write $h \in \Gamma(\Xi)$ to denote a shifted hyper-plane.

We denote the Fourier equivalence between two functions with the symbol ``$\fequiv$''. For example, the function $f(\bm{x})$ has the Fourier transform $\hat{f}(\bm\omega)$, which we also write as $f(\bm{x}) \fequiv \hat{f}(\bm\omega)$. We use the convention that the Fourier transform is defined distributionally as $$\hat{f}(\bm\omega) := \int f(\bm{x})\exp(-i\bm\omega\cdot \bm{x})\ d\bm{x}.$$
Of pivotal importance to us, is the Fourier representation of a box spline 
\begin{equation}
	\widehat{M_\bigxi}(\bm\omega) = \prod^n_{j=1} \frac{1 - \exp(-i\bm\omega\cdot\bm\xi_j)}{i\bm\omega\cdot\bm\xi_j}
\end{equation}
which we split into two functions, the difference operator 
\begin{equation}
	\widehat{\nabla_\bigxi}(\bm{\omega}) := \prod^n_{j=1} {1 - \exp(-i\bm\omega\cdot\bm\xi_j)}
\end{equation}
and the Green's function 
\begin{equation}
\widehat{\greens_\bigxi}(\bm{\omega}) := \prod^n_{j=1} {(i\bm\omega\cdot\bm\xi_j)^{-1}}
\end{equation}
which, when applied with $\widehat{D_\Xi} := \prod^n_{j=1} {i\bm\omega\cdot\bm\xi_j}$ produces the Dirac impulse.

To index an element of a vector $\bm{\beta}$ we write $\beta( j)$ where $j$ is an integer index.
We also define $w(j) := i\bm\xi_j\cdot\bm\omega$, and collect all of those terms in the $n$-dimensional variable $\bm{w}$. 

For the $\bm\alpha$-power function, we use the definition $\bm{x}^{\bm\alpha}:=\prod^m_{j=1}x(j)^{\alpha(j)}$ where $\bm{x}$ is an $m$-dimensional variable. 
It should be understood that by $\bm{x}^{-\bm\alpha}$ we explicitly mean $1/\bm{x}^{\bm\alpha}.$ 
We also adopt the notations 
\begin{eqnarray}
	\braket{\bm{x}}^{\bm\alpha} & := & \prod^m_{j=1}\frac{x(j)^{\alpha(j)}}{\alpha(j)!}, \\ 
	\braketsgn{\bm{x}}^{\bm\alpha} & := & \prod^m_{j=1}\frac{\left(x(j)\right)_{\sgn}^{\alpha(j)}}{\alpha(j)!}, \\ 
	\brakett{\bm{x}}^{\bm\alpha} & := & \prod^m_{j=1}\frac{\left(x(j)\right)_{+}^{\alpha(j)}}{\alpha(j)!} ,
\end{eqnarray}
for the normalized, normalized signed and normalized one-sided $\bm\alpha$-power functions respectively. 
In the above, $(x)^k_\sgn := x^k\sgn(x)/2$ and $(x)^k_{+} := x^kH(x)$ where $H(x)$ is the heaviside distribution. 
We write $||\bm\alpha||_0$ to count the non-zero elements within $\bm\alpha$. 
If $\bm\alpha$ is an $n$-dimensional vector, we can write the $s \times ||\bm\alpha||_0$ matrix $\bigxi_\alpha$, which denotes the subset of $\bigxi$ that are selected by the non-zero $\alpha(j)$. 
Finally we use $\bm{0}, \bm{1}$ to denote vectors consisting of all $0$ or $1$ respectively, and the dimension of the vector should be inferred from context in which it is being used.

The crux of this work is establishing explicit representations of the spatial forms of the difference operator and the Green's function of the box spline. 
The final piecewise polynomial form will result from the semi-discrete convolution of these two forms.
\section{Difference Operator} \label{sec:support}
The difference operator has a simple characterization in the spatial domain which can be constructed with the help of the following definition.

\begin{definition} \label{def:sxi}
For an $s \times n$ matrix $\Xi$ with full rank, define the sequence of multi-sets $S_\Xi$ by the rules  
\begin{equation}
	S_{\left[ \ \right]}:= \left\{(1,\bm{0})\right\} \text{ and } S_{\bigxi} := \left\{ (-c,\bm{\xi} + \bm{p})  : (c,\bm{p}) \in S_{\bigxi \backslash \bm{\xi}} \right\} \cup S_{\bigxi \backslash \bm{\xi} }.
\end{equation} 
For each tuple in each multi-set of this sequence, the first element of the tuple is a scalar value and the second element is a vector in $\mathbb{R}^s$. 
\end{definition}
From this definition we derive the following lemma.
\begin{lemma}
 \label{theorem:spat_diff}
Let $\bigxi$ be an $s \times n$ matrix with full rank, then the difference operator $\diffxi$ for $M_\Xi$ is given by the distribution
\begin{equation}
	\diffxi = \sum_{(c,\bm{p}) \in S_{\bigxi}} c \ \delta(\cdot - \bm{p})
\end{equation}
where $\delta$ is the Dirac delta distribution.
\end{lemma}

\begin{proof}
We show this inductively on the number of columns of $\bigxi$. 
As a base case, assume $\bigxi$ contains a single non-zero direction vector, $\bm\zeta$. By definition, we have $$S_\Xi = \{(1,\bm{0}), (-1, \bm\zeta) \}.$$
We also have  $\widehat{\nabla_{\bigxi}} = 1-\exp(i\bm\zeta\cdot\bm\omega)$ which has the spatial form $\diffxi = \delta - \delta(\cdot - \bm\zeta)$
which is by definition 
\begin{equation}
	\diffxi = \sum_{(c,\bm{p}) \in S_{\bigxi }} c \ \delta(\cdot - \bm{p}).
\end{equation}
For the general case, pick some $\bigxi$ with $n > 1$ columns, and suppose that for any $\bigxi^\prime$ with $i<n$ columns this result holds. 
Then split $\bigxi$ into an $s$ by $(n-1)$ matrix $\bigxi^\prime$ and column vector $\bm{\zeta}$, and consider the difference operator $\nabla_{\bigxi^\prime \cup \bm{\zeta}}$.
From the Fourier expression we have 
\begin{equation*}
	\widehat{\nabla_{\bigxi}} = \widehat{\nabla_{\bigxi^\prime \cup \bm{\zeta}}} = \prod_{\bm{\xi} \in \bigxi^\prime \cup \bm{\zeta}} (1 - \exp\left(-i \bm{\xi} \cdot \bm{\omega})\right) =  \left(1 - \exp(-i \bm{\zeta} \cdot \bm{\omega}) \right) \widehat{\nabla_{\bigxi^\prime}},
\end{equation*}
which is equivalently stated in the spatial domain as 
\begin{equation}
\diffxi = {\nabla}_{\bigxi^\prime \cup \bm{\zeta}} = \nabla_{\bigxi^\prime} - \nabla_{\bigxi^\prime}(\cdot - \bm{\zeta}).
\end{equation}
The inductive hypothesis then gives
\begin{equation}
{\nabla}_{\bigxi^\prime \cup \bm{\zeta}}= \sum_{(c,\bm{p}) \in S_{\bigxi^\prime}} c \  \delta(\cdot - \bm{p}) + \sum_{(c,\bm{p}) \in S_{\bigxi^\prime}}  -c \ \delta(\cdot - (\bm{p} + \bm\zeta) )
\end{equation}
and it follows, by definition of $S_{\bigxi}$, that
\begin{equation}
	\diffxi = {\nabla}_{\bigxi^\prime \cup \bm{\zeta}} = \sum_{(c,\bm{p}) \in S_{\bigxi }} c \ \delta(\cdot - \bm{p})
\end{equation}
which completes the induction.
\end{proof}

\section{Green's Function} \label{sec:greens}
The Green's function of a box spline allows us to characterize the polynomial within each region of the box spline's mesh. 
To do this we need the spatial form of the Green's function.
Going from the Fourier form to the spatial form of the Green's function is a challenge due to the general non-separability of the Fourier form. 
Here, we do this by pulling specially crafted vectors from $\ker\Xi$ that, when applied to the Fourier Green's function, can always reduce the Green's function to a separable form. We start with the definition of such vectors.

\begin{definition}
\label{def:nuvector}
Take $\Xi$ to be an $s \times n$ matrix with full rank, and take any $\bm\alpha$ with dimension $n$ and $s < ||\bm\alpha||_0 \le n$.
Let $\bm{b}^\prime_1 \cdots \bm{b}^\prime_n$ be the rows of $\ker\bigxi$, then construct an $(n-||\bm\alpha||_0) \times (n-s)$ matrix $C_\alpha$ where $\bm{b}_j^\prime$ is included in this matrix iff $\alpha(j) = 0$.
Let $\bm{t_\alpha}$ be the first column vector of $\ker{C_\alpha}$, then define $\bm{\nu_\alpha} := \ker\bigxi \cdot \bm{t_\alpha}$.
If $||\bm\alpha||_0 = n$ then such a vector does not necessarily exist, in this case we take $\bm{\nu_\alpha}$ as the first column vector of $\ker\Xi$.
\end{definition}
By construction $\bm{\nu_\alpha}$ has the following property.
\begin{proposition} \label{prop:zero_prop}
For $\bigxi$ and $\bm\alpha$ as above, the vector $\bm{\nu_\alpha}$ has the property that if $\alpha(j) = 0$ then $\nu_\alpha(j) = 0$.
\end{proposition}
\begin{proof}
This is immediate if $||\bm\alpha||_0 = n.$ So suppose that $s < ||\bm\alpha||_0 < n$, and suppose that $\alpha(j) = 0$ and $\nu_\alpha(j) \neq 0$. But then we have $\bm{b}^\prime_j\cdot\bm{t_\alpha}\neq 0$ so it must be that $\bm{t_\alpha} \notin \ker C_\alpha.$
\end{proof}
From this definition, we now construct a sequence of sets that will allow us to simplify the Green's function of any non-degenerate box spline into a separable form.
\begin{definition} \label{def:pns}
Let $\Xi$ be an $s \times n$ matrix with full rank, define the sequence of sets $P$ recursively by the rules
\begin{eqnarray}
	P_0 & := & \left\{ (1,\bm{1}) \right\} \\
	P_k & := & \left\{ \left(-c\frac{ \nu_{\alpha}(j) }{\nu_{\alpha}(m_{\alpha})}, \bm{r}_j(\bm\alpha) \right) : (c,\bm\alpha) \in P_{k-1}, m_{\alpha}<j\le n , \nu_\alpha(j)\ne 0 \right\} \label{eqn:r_const} 
\end{eqnarray}
with $m_{\alpha}:= \min\{i : \nu_\alpha(i) \ne 0\}$ and the vector function $\bm{r}_j$ defined as 
\begin{equation}
	\bm{r}_j:\zahlen^n \rightarrow \zahlen^n : \alpha(i) \mapsto \left\{
     \begin{array}{lcl}
       \alpha(i) + 1 & , & \text{ if } i = m_{\alpha}   \\
       \alpha(i) - 1 & , & \text{ if } i = j  \\
       \alpha(i) & ,  & \text{ otherwise }
     \end{array}
   \right.
\end{equation}
and $k \le n-s.$ Here, each $P_i$ is a set of tuples where the first element of each tuple is a real scalar value, and the second element is a vector in $\mathbb{R}^n$.
\end{definition}
Before we use this set to decompose the Green's function of $M_\Xi$, we need some supplementary results.

\newpage
\begin{proposition} \label{prop:abound}
If $(c,\bm\alpha) \in P_k$, then for all $l$ such that $m_\alpha < l \le n,$ we have 
\begin{equation}
	\alpha(l) \in \{0,1\}.
\end{equation}
\end{proposition}
\begin{proof}
If $k=0$, then this is true by definition.
Next, for the inductive case, suppose that this is true for all $j$ such that $0 \le j < k$, and pick any $(c,\bm\alpha) \in P_k$. 
Since $k>0$, there exists a $(d,\bm\beta)\in P_{k-1}$ such that
$
	c = -d\frac{\nu_\beta(l)}{\nu_\beta({m_{\beta}})}, \text{ and } \bm\alpha = \bm{r}_l(\bm\beta),
$
for some $l$ and $m_\beta$.
If $m_\alpha \ge m_\beta,$ then $\beta(i)\in\{0,1\}$, for all $i$ such that $m_\beta < i \le n$, but by construction in equation (\ref{eqn:r_const}), we know 
$\alpha(i)\in\{0,1\}$ for all $m_\beta < i \le n$. But then $\alpha(i) \in \{0,1\}$ for all $m_\alpha < i \le n$, which completes the proof. So suppose that $m_\alpha > m_\beta$. Recall by construction, we have
$\nu_\alpha(m_\alpha) \ne 0$ and $\nu_\beta(m_\beta) \ne 0$. Combine this fact with the definition $\bm{\nu_\alpha} := \ker\Xi \bm{t_\alpha}$,  $\bm{\nu_\beta} := \ker\Xi \bm{t_\beta}$, and with the fact that $\ker\Xi$ is in column echelon form, we have 
\begin{equation} \label{eqn:lower}
	\min\{i:t_\alpha(i) \ne 0\} < \min\{i : t_\beta(i) \ne 0\}.
\end{equation}
Since $\bm{t_\beta} \in \ker C_\beta$ forces certain elements of $\bm\beta$ to be zero, and $\bm{t_\alpha}$ forces the same elements of $\bm\alpha$ to be zero (and perhaps more), we know $\bm{t_\alpha} \in \ker C_\beta$ as well. But then $\bm{t_\beta}$ could not have been the first column of $\ker C_\beta$, since there exists a column with a lower-indexed, non zero row (by equation (\ref{eqn:lower})), and $\ker C_\beta$ is in column echelon form. 

There is however, one case left to show, since for $k=1$ no such $\bm{t_\beta}$ exists. Thus if $k=1$, choose any $(c,\bm\alpha) \in P_1$, and let $(d, \bm\beta) \in P_0$ (chosen as above). Again, by the same logic as above, if $m_\alpha \ge m_\beta,$ then the proof is completed. So suppose that $m_\alpha < m_\beta$. But since both $\bm{\nu_\alpha} \in \ker\Xi$ and $\bm{\nu_\beta}\in\ker\Xi$, $\nu_\alpha(m_\alpha) \ne 0$ and $\nu(m_\beta) \ne 0$ and $\ker\Xi$ is in row echelon form, this contradicts our choice of $\bm{\nu_\alpha}$ since it could not have been the first column of $\ker\Xi$.

\end{proof}

\begin{proposition}\label{prop:dmonovariate}
If $(d,\bm\beta) \in P_k$ for $0\le k < n - s$, and $(c, \bm\alpha) \in P_{k+1}$, with 
\begin{equation}
	c = -d\frac{\nu_\beta(j)}{\nu_\beta({m_{\beta}})}, \text{ and } \bm\alpha = r_j(\bm\beta)
\end{equation}
for some $m_\beta <j \le n$, then
	\begin{equation}
		\znorm{\bm\alpha} = \znorm{\bm\beta} - 1.
	\end{equation}
\end{proposition}
\begin{proof}
We show this directly. Choose any $(d,\bm\beta) \in P_k$ and $(c, \bm\alpha) \in P_{k+1}$ as stated in the proposition. Then, by Proposition (\ref{prop:abound}), we know that $\beta(i) \in \{0,1\}$ for all $i$ such that $m_\alpha < i \le n$, by construction in equation (\ref{eqn:r_const}) exactly one of these non-zero $\beta(i)$ will be decremented (and $\beta(m_\beta)$ will be incremented) to become $\bm\alpha$, thus $\znorm{\bm\alpha} = \znorm{\bm\beta} - 1 $.

\end{proof}

\begin{proposition} \label{prop:rank}
If $(c,\bm\alpha) \in P_k, 0\le k \le n$, then $\Xi_\alpha$ has full rank.
\end{proposition}
\begin{proof}
For $P_0$ this is holds by assumption. 
Next, assume this holds for all $(d,\bm\beta) \in P_j$ where $0\le j < k$. 
Pick any $(c, \bm\alpha) \in P_k.$
There exists some $(d,\bm\beta) \in P_{k-1}$, such that $\bm\alpha = \bm{r}_l(\bm\beta)$ with $m_\beta < l \le n$, for some $m_\beta$ and $l$.
By construction of $\bm\alpha$ in equation (\ref{eqn:r_const}), we clearly have either $\Xi_\beta = \Xi_\alpha$, or $\Xi_\beta \ne \Xi_\alpha$. If $\Xi_\beta = \Xi_\alpha$, then the proof is complete, so suppose that $\Xi_\beta \ne \Xi_\alpha$. By Proposition \ref{prop:dmonovariate}, these differ by exactly one column vector, specifically, equation (\ref{eqn:r_const}) tells us that if $\bm\xi_l \in \Xi_\beta$, then $\bm\xi_l \notin \Xi_\alpha$. It remains to show that $\bm\xi_l$ can be written as a linear combination of vectors only in $\Xi_\beta$ (not including $\bm\xi_l$ itself). But we also know
\begin{equation}
	\sum_{i=1}^n\bm\xi_i \nu_\beta(i)   =  \bm{0}
\end{equation}
since $\bm{\nu_\beta} \in \ker\Xi$. By Proposition \ref{prop:zero_prop}, we can write this
as
\begin{equation}
	 \sum_{i=1,i\ne l}^n\bm\xi_i \nu_\beta(i)  = -\nu_\beta(l)\bm\xi_l.
\end{equation}
Note that all the non-zero elements of $\bm{\nu_\beta}$ correspond to column vectors of $\Xi_\beta$ thus we can write $\bm\xi_l$ as the weighted sum of vectors from $\Xi_\beta.$ 
Since $\Xi_\beta$ had full rank, $\Xi_\alpha$ has full rank, completing the induction. 
\end{proof}

\begin{lemma} \label{lemma:decomp1}
Let $\bigxi$ be an $s \times n$ matrix with full rank, then 
\begin{equation}
	\widehat{G_\bigxi}(\bm\omega) = \sum_{(c,\bm\alpha) \in P_k} c \ \bm{w}^{-\bm\alpha}
\end{equation}
for $0 \le k \le n - s.$
\end{lemma}
\begin{proof}
When $k=0$ the result is true by definition. 
Suppose the result holds for all $j$ such that  for $0 \le j < k$. 
By assumption we have 
\begin{equation}
	\widehat{G_\bigxi}(\bm\omega) = \sum_{(c,\bm\alpha) \in P_{k-1}} c \ \bm{w}^{-\bm\alpha}.
\end{equation}
Next, since $\bm\nu_{\bm\alpha} \in \ker\bigxi$ we have $\bm\nu_{\bm\alpha} \cdot \bm{w} = 0$, re-arranged this is
\begin{equation}
\frac{\nu_\alpha(m_{\alpha})w(m_{\alpha})-\bm{\nu}_{\alpha} \cdot \bm{w}}{\nu_\alpha(m_{\alpha})w(m_{\alpha})} = 1.
\end{equation}
So we certainly have 
\begin{eqnarray}
\widehat{G_\bigxi}(\bm\omega) & = & \sum_{(c,\bm\alpha) \in P_{k-1}} (c \ \bm{w}^{-\bm\alpha}) \frac{\nu_\alpha(m_{\alpha})w(m_{\alpha})-\bm{\nu}_{\bm\alpha} \cdot \bm{w}}{\nu_\alpha(m_{\alpha})w(m_{\alpha})} \\ 
 & = & \sum_{(c,\bm\alpha) \in P_{k-1}} \ \sum_{l=m_{\alpha}+1}^n -c \frac{\nu_{\alpha}(l)}{\nu_\alpha(m_{\alpha})} \bm{w}^{-\bm{r}_l(\bm\alpha)}
\end{eqnarray}
which is, by definition,
\begin{equation}
	\sum_{(c,\bm\alpha) \in P_{k}} c \ \bm{w}^{-\bm\alpha},
\end{equation}
completing the induction. 
\end{proof}

The next definition helps project $n$-dimensional vectors into appropriate $s$-dimensional vectors, which will aid in the transformation from Fourier form to the final spatial form.
\begin{definition}
\label{def:muvector}
For any $\bm\alpha$ of dimension $n$ with $\znorm{\bm\alpha}=s$ define $\bpsquash{\alpha}$  as the $s$-dimensional vector obtained by retaining the non-zero elements of $\bm\alpha$. Define also the corresponding vector $\bidxsquash{\alpha}$ that indexes the non-zero elements of $\bm\alpha$.
\end{definition}
For example, if $s=2$ and $\bm\alpha = (0,2,0,1)$, then $\bpsquash{\alpha}=(2,1)$ and $\bidxsquash{\alpha} = (2,4)$. With this, we now define a helpful auxiliary function.
\begin{definition}
\label{def:transformaux}
For any $\bm\alpha$ with dimension $n$ and $\znorm{\bm\alpha}=s$, define the function
$$ 
T_{\alpha}(\bm{\omega}) := \prod^s_{j=1}\frac{1}{i\omega(j)^{{\mu_\alpha}(j)}}.
$$
\end{definition}

\begin{proposition}
For any $\bm\alpha$ with dimension $n$ and $\znorm{\bm\alpha}=s$ then
$$
T_{\bm\alpha}(\bm{\omega}) \fequiv \braketsgn{\bm{x}}^{\bm{\mu_\alpha - 1}}.
$$
\end{proposition}
\begin{proof}
Using the Fourier equivalence
\begin{equation}
\frac{1}{(i\omega )^k} \fequiv \frac{(x)_\sgn^{k-1}}{(k-1)!}, \  k \in \zahlen^+,
\end{equation}
we write 
$$ 
T_{\alpha}(\bm{\omega}) = \prod^s_{j=1}\frac{1}{i\omega(j)^{{\mu_\alpha}(j)}} \fequiv \prod^s_{j=1}\frac{(x(j))_\sgn^{{{\mu_\alpha}(j)} - 1}}{({{\mu_\alpha}(j)}-1)!} = \braketsgn{\bm{x}}^{\bm{\mu_\alpha - 1}}
$$
by definition. 
\end{proof}

\begin{proposition} \label{prop:xform}
For any $s \times n$ matrix $\bigxi$ with full rank, and any $\bm\alpha$ with dimension $n$ and $\znorm{\bm\alpha}=s$ then
$$
\bm{w}^{-\bm\alpha} \fequiv \frac{\braketsgn{\bigxi^{-1}_{\alpha}\bm{x}}^{\bm{\mu_\alpha - 1}}}{|\det\bigxi_{\alpha}|}.
$$
\end{proposition}
\begin{proof}
Since $\znorm{\bm\alpha}=s$ we can write 
\begin{eqnarray}
    \bm{w}^{-\bm\alpha} & = & \prod^s_{j=1}\frac{1}{w({{\sigma_{\alpha}}(j)})^{{\mu_{\alpha}}(j)}} \\ 
    & = & \prod^s_{j=1} (i\bm\xi_{{\sigma_{\alpha}}(j)}\cdot\bm\omega)^{-{\mu_{\alpha}}(j)} \\
    & = & T_{\alpha}(\bigxi^T_{\alpha}\bm\omega),
\end{eqnarray}
and    
\begin{equation}
    T_{\alpha}(\bigxi^T_{\alpha}\bm\omega)  \fequiv  |\det\bigxi^{-1}_{\alpha}|\braketsgn{\bigxi^{-1}_{\alpha}\bm{x}}^{\bm{\mu_\alpha - 1}}.
\end{equation}
\end{proof}

\begin{lemma}\label{theorem:spat_green}
Let $\bigxi$ be an $s \times n$ matrix with full rank, then 
\begin{equation}
	G_\bigxi(\bm{x}) = \sum_{(c,\bm\alpha) \in P_{n-s}} \frac{c}{|\det\bigxi_{\alpha}|}\braketsgn{\bigxi_{\alpha}^{-1}\bm{x}}^{(\bpsquash{\alpha}-\bm{1})}.
\end{equation}
\end{lemma}
\begin{proof}
Starting with
\begin{equation}
	\widehat{G_\bigxi}(\bm\omega) = \sum_{(c,\bm\alpha) \in P_{n-s}} c\bm{w}^{-\bm\alpha}.
\end{equation}
then using Propositions \ref{prop:dmonovariate} and \ref{prop:xform} we obtain the form stated in the lemma, the only thing left to check is that this is well defined, i.e. we must verify that $\det\bigxi_{\alpha} \ne 0$, but this is true by Proposition \ref{prop:rank}. 
\end{proof}
Combining all these small results leads us to the main result of this work.
\begin{theorem}
Let $\bigxi$ be an $s \times n$ matrix that corresponds to a non-degenerate box spline, then the box spline is given by 
\begin{equation} \label{eq:final_form}
	M_\bigxi(\bm{x}) = \sum_{(b,\bm{p})\in S_\bigxi} \sum_{(c,\bm\alpha) \in P_{n-s}} \frac{b \cdot c}{\left|\det\bigxi_{\alpha}\right|} \brakett{(\bigxi_{\alpha}^{-1}\bm{x}) - \bm{p}}^{(\bpsquash{\alpha}-\bm{1})}.
\end{equation}
\end{theorem}
\begin{proof}
This is a direct consequence of Lemma \ref{theorem:spat_diff} and \ref{theorem:spat_green}, combined with the fact that the difference operator annihilates all polynomials of degree lower than the order of the spline \cite[I.32]{deboorbox}. 
\end{proof}

\section{Fast and Stable Evaluation} \label{sec:stable_eval}
While the form in equation (\ref{eq:final_form}) is explicit, it is important to note that it is a distributional equality.
Thus, point-wise equality does not necessarily hold. 
However, when the function is evaluated away from the knot planes, the function will be a polynomial locally, and the evaluation will be numerically stable. 
Issues only arise when one wishes to evaluate (\ref{eq:final_form}) for points on the knot planes, which may be undefined.

To overcome this, we use the knot planes to find the regions of evaluation for the box spline.
We take the polyhedron defined by the direction vectors of the box spline, then split that support of the spline with the knot planes, this produces a list of all the regions to which we can assign polynomials.
The polynomial within each region $R$ can then be obtained as 
\begin{equation} \label{eq:poly_form}
	P_R(\bm{x}) := \sum_{(b,\bm{p})\in S_\bigxi} \sum_{(c,\bm\alpha) \in P_{n-s}} \frac{ b \cdot c } {\left|\det\bigxi_{\bm\alpha}\right|} \brakett{(\bigxi_{\bm\alpha}^{-1}\bm{c}_R) - \bm{p}}^{(\bpsquash{\alpha}-\bm{1})}  \left(({\bigxi_{\bm\alpha}^{-1}\bm{x}) - \bm{p}}\right)^{(\bpsquash{\alpha}-\bm{1})}
\end{equation}
where $\bm{c}_R$ is the center of the region $R$. 
This can be interpreted as selecting all the Green's polynomials that touch $R$ in the convolution (\ref{eq:final_form}), giving a polynomial for only that region.

If a point lies on a knot plane, and the box spline is continuous at that point, then we can simply choose one of the surrounding regions and evaluate $P_R(\bm{x})$ at that point, since the continuity of the spline tells us that it does not matter which polynomial piece we choose at boundaries. 
To obtain a fast evaluation scheme, we calculate all regions $R$ and their corresponding $P_R(\bm{x})$, 
then use the knot-planes to define a binary space partitioning (BSP) tree.
The leaves of this tree will correspond to the regions $R$, thus we associate $P_R(\bm{x})$ to each leaf. 
The computation of this tree is outlined in Algorithm (\ref{alg:precompute-bs}) and the evaluation at a point is outlined in Algorithm (\ref{alg:eval-bs}).
In the following algorithms there are two object classes, $InternalNode()$ and $LeafNode()$ that inherit from a base $Node()$ class. Additionally, whenever we speak of some hyperplane $h$, we speak of it as the normal equation $h(\bm{x}) = \bm{n}\cdot\bm{x} - d$ for whatever vector $\bm{n}$ and $d$ define that plane.

\begin{algorithm} 
\caption{Precompute box spline pieces}\label{alg:precompute-bs}
\begin{algorithmic}[1]
\Procedure{PrecomputeBoxSpline}{$\Xi$}
	\State /* Returns a tuple containing the polyhedron of the support of $\Xi$ 
	\State as the first element and the root Node object as the second*/
	\State $\textit{Compute } S_\Xi \textit{ from Definition \ref{def:sxi}}$
	\State $\textit{Compute } P_{n-s} \textit{ from Definition \ref{def:pns}}$
	\State $\textit{Compute knot-planes }h \in \Gamma(\Xi) \textit{ such that } h \cap supp(\Xi)\textit{. Collect these in } H$
		\State $\textit{Calculate the polyhedron } Q \textit{ for } supp(\Xi)$
	\State \Return $(Q, Recurse(H, Q))$
	
\EndProcedure
\Procedure{Recurse}{H, Q}
	\State /* Returns a Node object */
	\State $\textit{Choose } h \in H \textit{ that splits } Q $
	
	\If{$\textit{no such } h \textit{ exists }$}
		\State $\textit{Use equation \ref{eq:poly_form} to derive a polynomial peice } R \textit{ at the center of Q}$
		\State \Return $LeafNode(R)$
	\EndIf
	\State $\textit{Split } Q \textit{ with } h \textit{ into left and right regions } A,B$
	\State /* Here, left means all points $\bm{x}$ that satisfy $h(x)\le 0$, right is the opposite */
	\State $node \gets InternalNode()$
	\State $node.h \gets h$
	\State $node.left \gets Recurse(H \backslash \{ h \}, A)$
	\State $node.right \gets Recurse(H \backslash \{ h \}, B)$
	\State \Return $node$
\EndProcedure
\end{algorithmic}
\end{algorithm}

\begin{algorithm}
\caption{Evaluate box spline}\label{alg:eval-bs}
\begin{algorithmic}[1]
\Procedure{EvalBoxSpline}{$\bm{x}$, $Q$, $root$}
	\If{$\bm{x} \notin Q$}
		\State \Return 0
	\EndIf

	\State \Return $RecurseEval(\bm{x}, root)$
\EndProcedure

\Procedure{RecurseEval}{$\bm{x}$, $node$}
	\If{$node \textit{ is a LeafNode }$ }
		\State \Return $\textit{ evaluate node's polynomial at } \bm{x}$
	\EndIf
	\State $h \gets node.h$
	\State $L \gets node.left$ 
	\State $R \gets node.right$ 
	\If{ $h(\bm{x}) \le 0$}
		\State \Return $Recurse(\bm{x}, L)$
	\EndIf
	\State \Return $Recurse(\bm{x}, R)$
\EndProcedure
\end{algorithmic}
\end{algorithm}

It is important to note that the choice of $h$ in Algorithm \ref{alg:precompute-bs} will affect the performance of evaluation; $h$ should be chosen so as to balance the resulting binary tree. As a heuristic, we choose $h$ so that it splits the given polyhedron into polyhedra of roughly equivalent hyper-volumes. On average, the average traversal time should be $O(log(k))$ where $k$ is the total number of regions of the box spline. 
A variant on this algorithm is implemented in a SageMath  \cite{sage} worksheet and is available as supplementary material.
\section{Examples} \label{sec:examples}
\def\arraystretch{1.1}
Here we enumerate some examples of box splines. When feasible, we enumerate their polynomial regions (see Tables (1-3)), and  $P$ and $S$ sets. When we enumerate the sets for these box splines, we enumerate simplified versions of the sets, combining elements that correspond to like terms in the final summation.

\begin{example}{Courant Element:} \label{ex:courant}
The first example we consider is the well known Courant element which is constructed by taking the the two Cartesian principle lattice directions and adding the diagonal vector (see Figure (\ref{fig:courant})). This is given by the direction matrix
\begin{equation}
  \Xi = \begin{bmatrix}
    1 & 0 & 1 \\
    0 & 1 & 1
  \end{bmatrix}.
\end{equation}

\begin{figure}[h]
	\centering
	\includegraphics[scale=0.250]{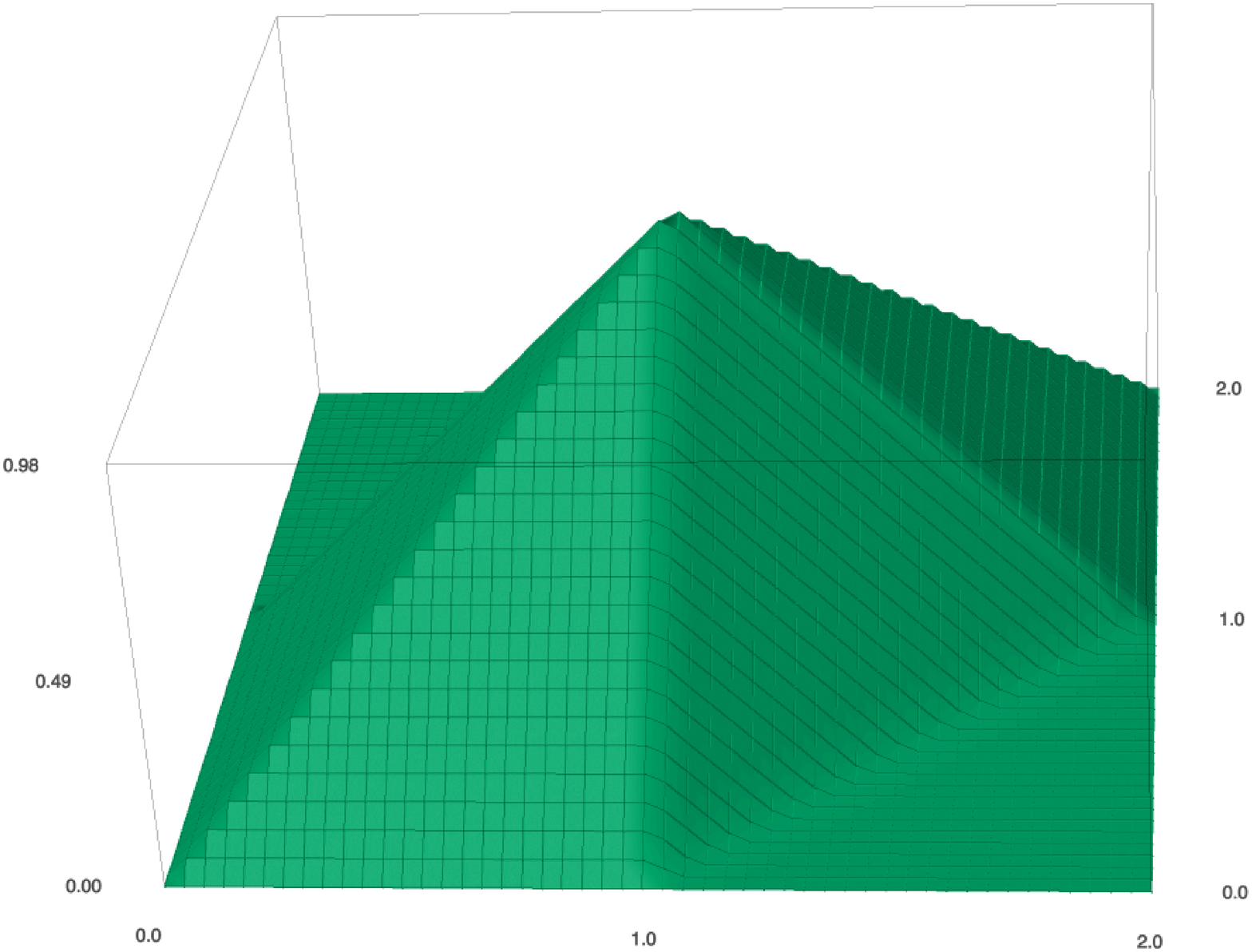}
	\includegraphics[scale=0.35]{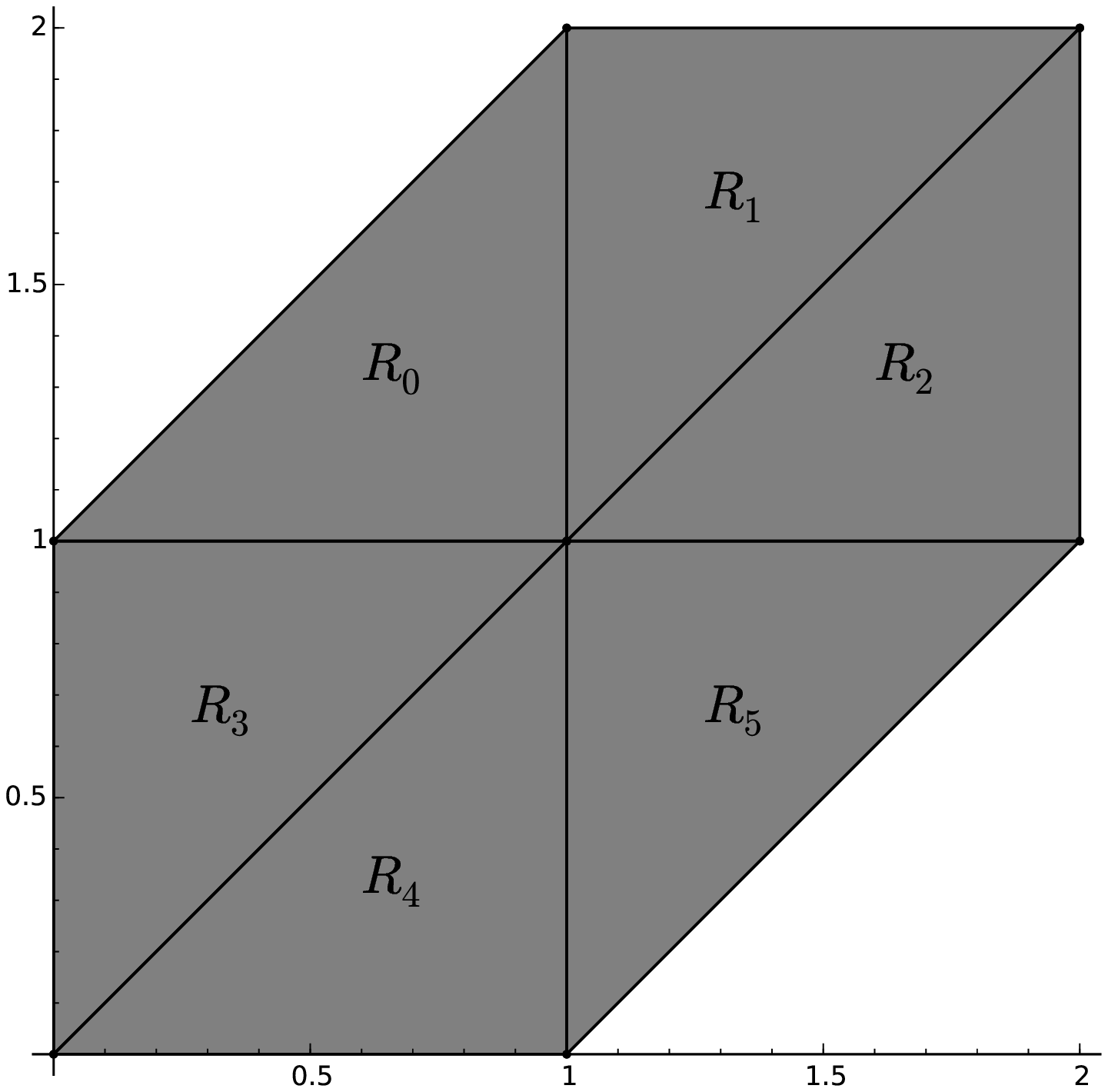}

	\caption{On the left is a plot of the Courant element, on the right is a plot of the different polynomial regions of the box spline.}
	\label{fig:courant}
\end{figure}

Enumerating the sets for this direction matrix yields 
  \begin{eqnarray*}
  S_\Xi  & = & \left\{
    \left(-1, \left( 2,  2 \right) \right),
    \left(1, \left( 2, 1 \right) \right),
    \left(1, \left( 1, 2 \right) \right),
    \left(-1, \left( 1,  0 \right) \right),
    \left(-1, \left( 0, 1 \right) \right),
    \left(0, \left( 1,  2 \right) \right)
  \right\}, \\
  P_1 & = & \left\{
    \left(1, \left(2, 1, 0 \right) \right),
    \left(-1, \left( 2, 0, 1 \right) \right)
  \right\}.
  \end{eqnarray*}

\end{example}

\begin{example}{Zwart-Powell Element:} \label{ex:zpelement}

The Zwart-Powell element is similar to the Courant element but has an additional vector added along a complementary diagonal. We provide a more thorough example for the construction of $P_2$ for the ZP element (see Figure (\ref{fig:zpel})). We first start by noting the ZP element's direction matrix and a basis for its null-space
  
\begin{equation}
    \Xi = \begin{bmatrix}
      1 & 0 & 1 & -1\\
      0 & 1 & 1 & 1
    \end{bmatrix}, \ \ker \Xi = \begin{bmatrix}
     \ 1 & \ 0 \\
     \ 1 & \ 2 \\
     -1 & -1 \\
     \ 0 & -1 
    \end{bmatrix}.
\end{equation}
We first construct $P_1$ by definition (\ref{def:pns}). Consider $(c,\bm\alpha) = (1, (1,1,1,1)) \in P_0$. Since this is the base case in definition (\ref{def:nuvector}), for this $\bm\alpha$ we have
$$
\bm{\nu_\alpha} = (1,1,-1, 0), \text{ and } m_\alpha = 1.
$$
Definition (\ref{def:pns}) also dictates that this produces elements for all $j$ with $m_\alpha < j \le n$ such that $\nu_\alpha(j)\ne 0$. For $j=2$ and $j=3$ we have the tuples
$$
\left(-1 \frac{1}{1}, (2,0,1,1)\right) \text{ and } \left(-1\frac{-1}{1}, (2,1,0,1)\right)
$$
respectively. Thus we have $P_1 = \{ (-1, (2,0,1,1)), (1, (2,1,0,1)\}$, and we apply the recursion to each tuple of this multiset to construct $P_2$. For $(c,\bm\alpha) = (-1, (2,0,1,1)) \in P_1$ we construct 
$$
C_\alpha = \begin{bmatrix}
1 & 2
\end{bmatrix}
\text{ which has } \ker C_\alpha = \begin{bmatrix}
 2 \\ 
 -1
\end{bmatrix}.
$$
Thus $\bm{t_\alpha} = (2, -1)$, and we have 
$$
\bm{\nu_\alpha} = \ker\Xi \cdot \bm{t_\alpha} = (2,0,-1, 1), \text{ and } m_\alpha = 1,
$$
which produces the elements 
$$
\left(-(-1) \frac{-1}{2}, (3,0,0,1)\right) \text{ and } \left(-1\frac{-1}{2}, (3,0,1,0)\right)
$$
for $j=3$ and $j=4$ respectively. Finally, we look at the case $(c,\bm\alpha) = (1, (2,1,0,1)) \in P_1$. Again, we construct 
$$
C_\alpha = \begin{bmatrix}
-1 & -1
\end{bmatrix}
\text{ which has } \ker C_\alpha = \begin{bmatrix}
 1 \\ 
 -1
\end{bmatrix}.
$$
This gives $\bm{t_\alpha} = (1, -1)$, and we have 
$$
\bm{\nu_\alpha} = \ker\Xi \cdot \bm{t_\alpha} = (1,-1,0, 1), \text{ and } m_\alpha = 1,
$$ producing the tuples 
$$
\left(-1 \frac{-1}{1}, (3,0,0,1)\right) \text{ and } \left(-1\frac{1}{1}, (3,1,0,0)\right).
$$
So we have 
$$
P_2 = \left\{ 
\left(-1/2, (3,0,0,1)\right), \left(1/2, (3,0,1,0)\right), \left(1, (3,0,0,1)\right), \left(-1, (3,1,0,0)\right) \right\}
$$
or, equivalently (by collecting like terms)
$$
P_2 = \left\{ 
\left(1/2, (3,0,0,1)\right), \left(1/2, (3,0,1,0)\right), \left(-1, (3,1,0,0)\right) \right\}.
$$ Applying the recursion in definition (\ref{def:sxi}) also yields the set
  \begin{eqnarray*}
    S_\Xi  & = & \{
    \left(-1, \left(-1/2,3/2\right)\right),
    \left(1, \left(-1/2,-3/2\right)\right),
    \left(1, \left(-3/2,1/2\right)\right),\left(1, \left(1/2,3/2\right)\right),\\
    & & 
    \left(-1, \left(1/2,-3/2\right)\right),
    \left(-1, \left(-3/2,-1/2\right)\right),
    \left(1, \left(3/2,-1/2\right)\right),
    \left(-1, \left(3/2,1/2\right)\right) \}.
    \end{eqnarray*}
  
	\begin{figure}[h]
		\centering
		\includegraphics[scale=0.200]{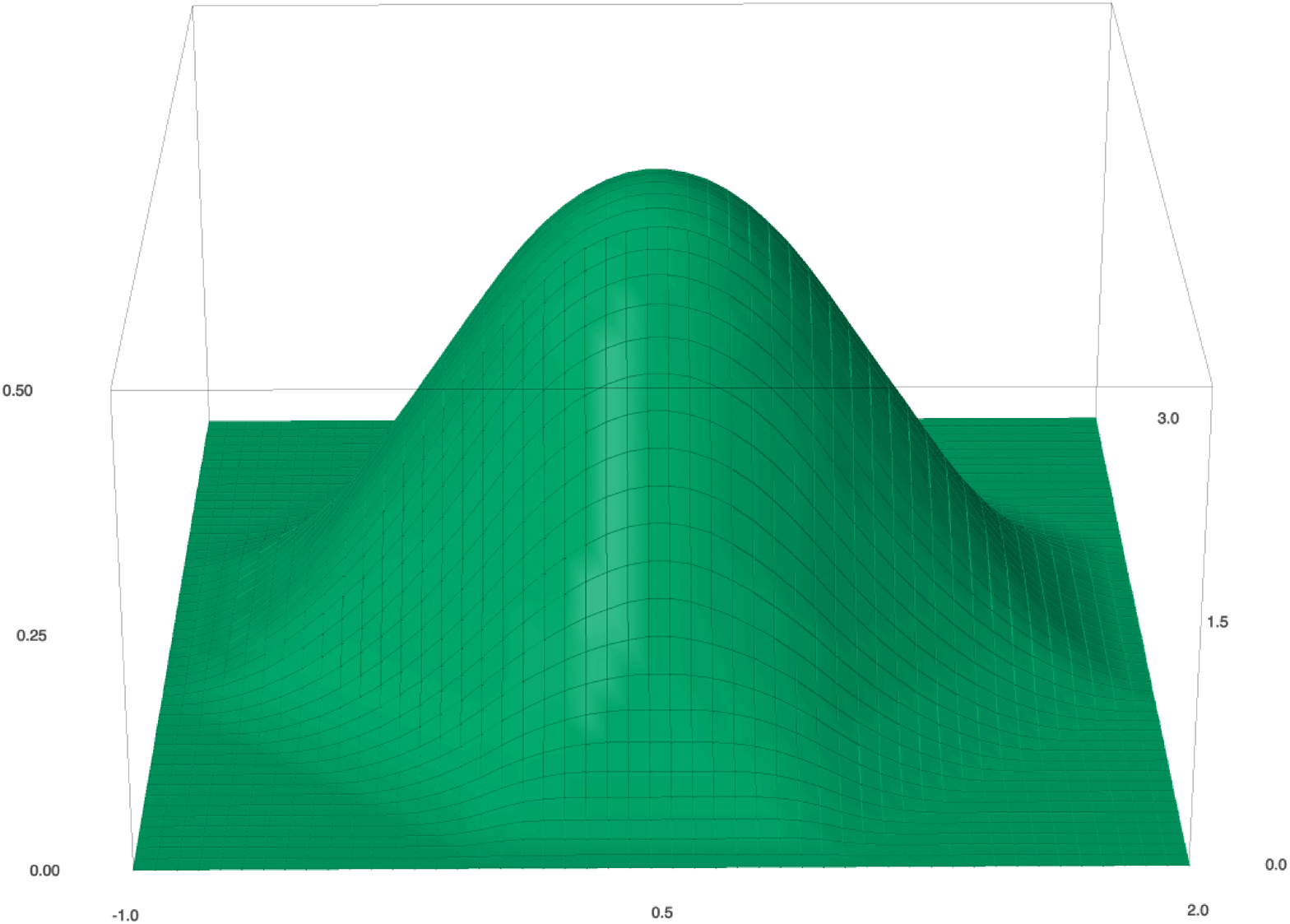}
		\includegraphics[scale=0.25]{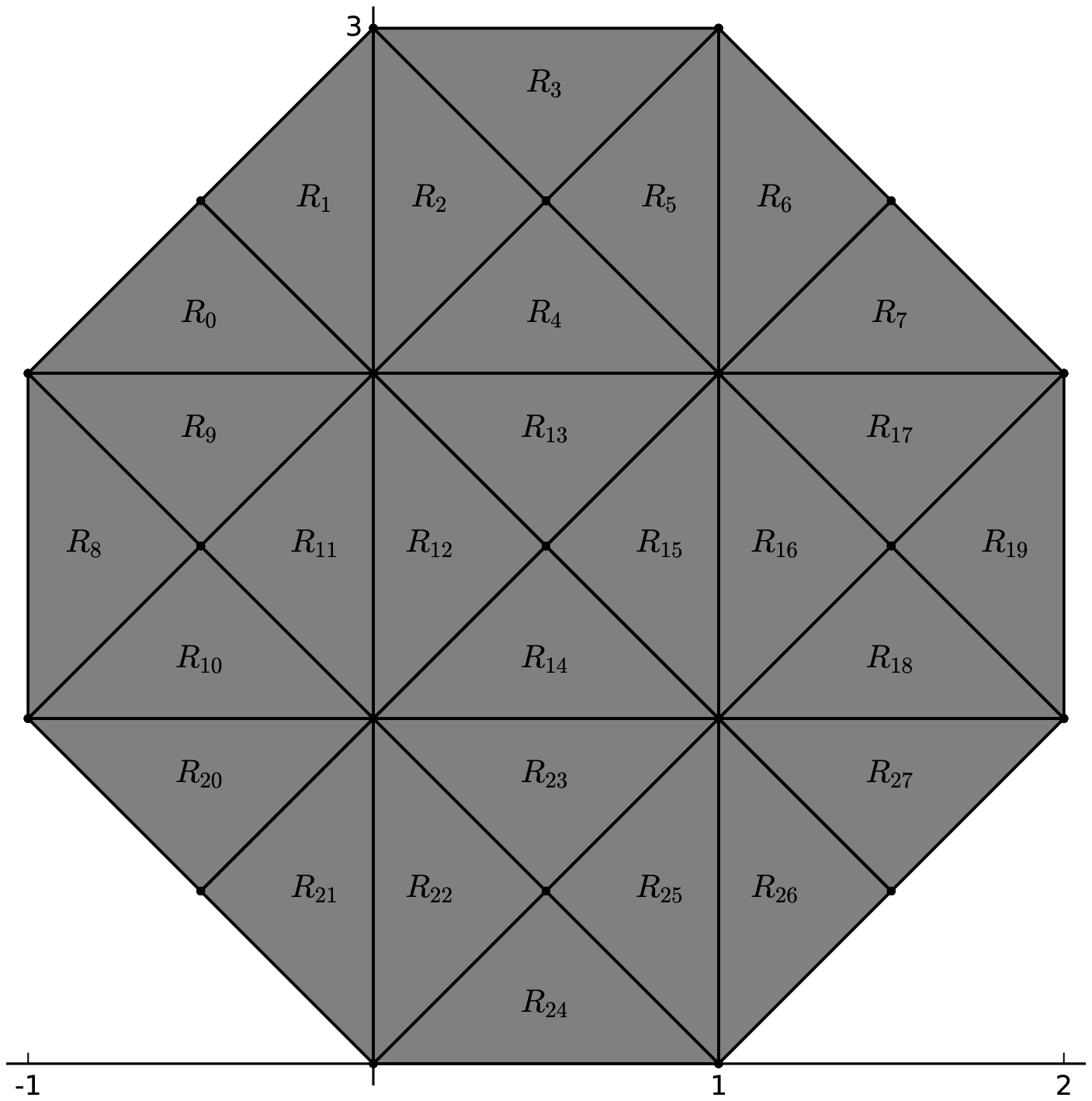}
		\caption{On the left is a plot of the ZP-element, on the right is a plot of the different polynomial regions of the box spline.} \label{fig:zpel}
	\end{figure}

\end{example}
\newpage
\begin{example}{Skewed Element:} \label{ex:potato}
As an example of a box spline that has not appeared in the literature, we introduce the four direction skewed element (see Figure (\ref{fig:potato})) defined by the matrix
  \begin{equation}
      \Xi = \begin{bmatrix}
      1 & 0 & 1 & 2\\
      0 & 1 & 1 & 1
    \end{bmatrix}.
  \end{equation}

  \begin{figure}[!ht]
    \centering
    \includegraphics[scale=0.190]{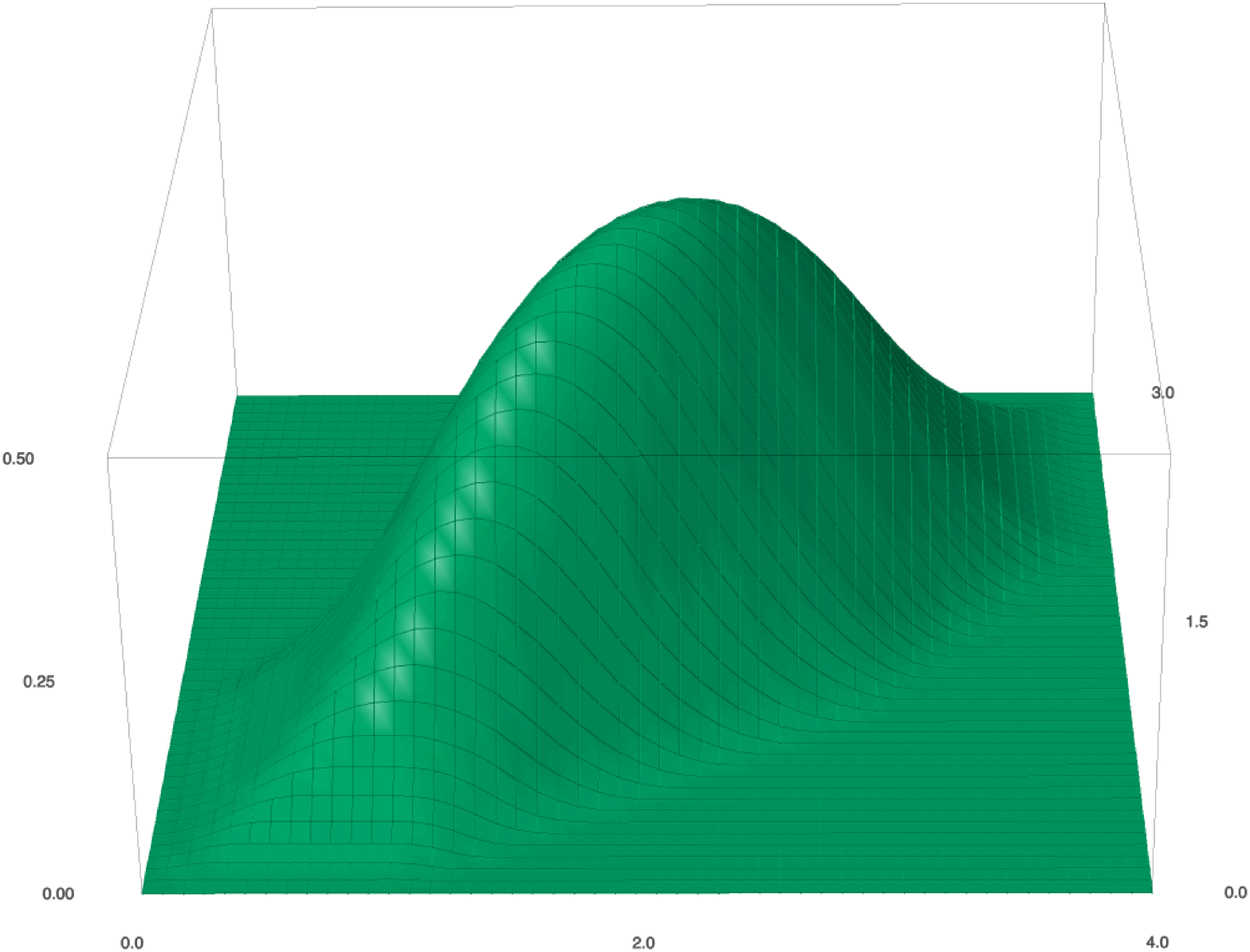}
    \includegraphics[scale=0.35]{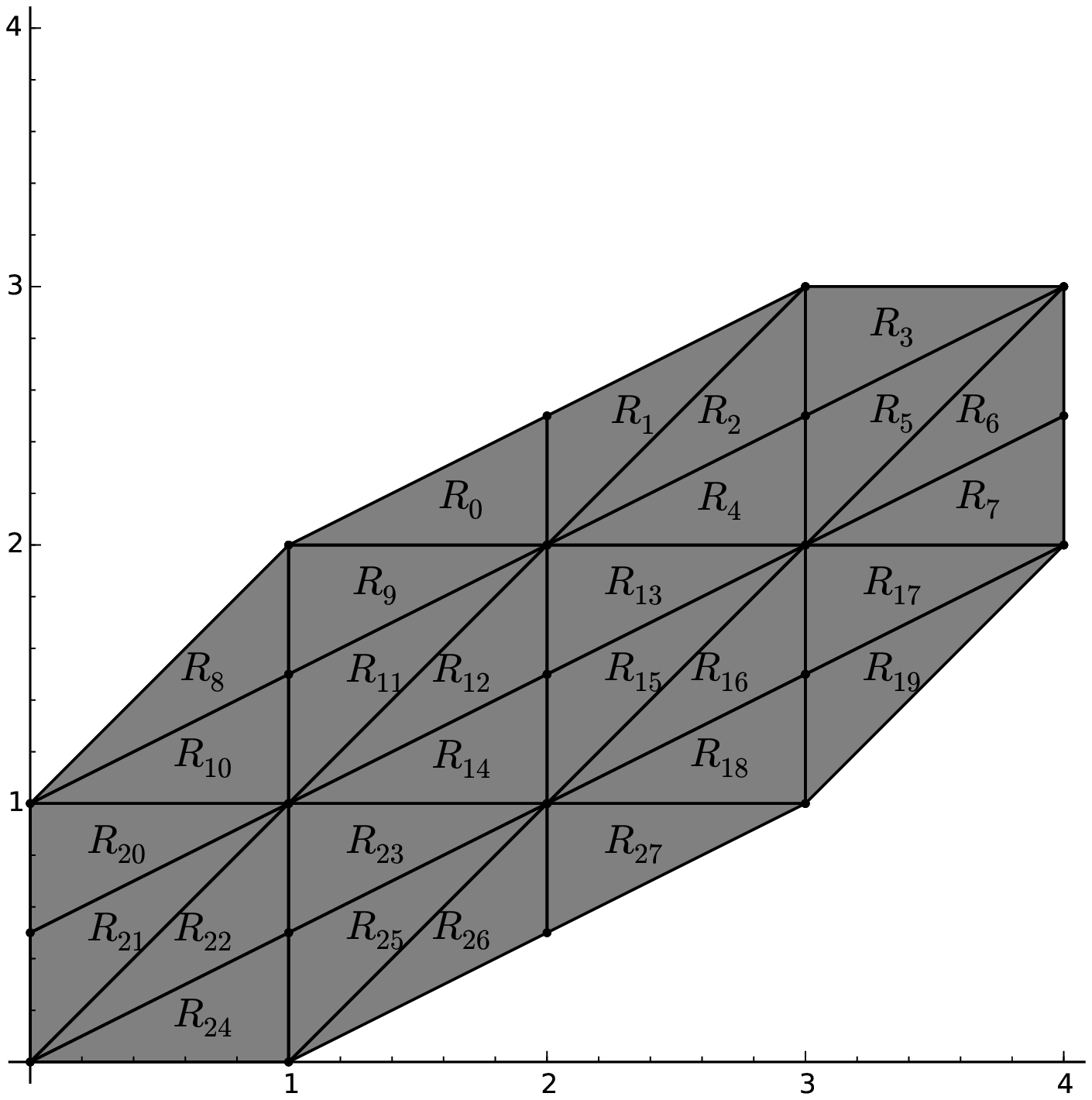}
    \caption{On the left is a plot of the Skewed element, on the right is a plot of the different polynomial regions of the box spline.} \label{fig:potato}
  \end{figure}

  \begin{eqnarray*}
    S_\Xi  & = & \{
      \left(1, \left(2,3/2\right)\right),
      \left(-1, \left(2,1/2\right)\right),
      \left(-1, \left(-2,-1/2\right)\right),
      \left(1, \left(-2,-3/2\right)\right),\\ & &
      \left(1, \left(-1,1/2\right)\right),
      \left(1, \left(1,-1/2\right)\right),
      \left(-1, \left(1,3/2\right)\right),
      \left(-1, \left(-1,-3/2\right)\right)      
    \} \\
    P_2 & = & \left\{
      \left(1/2, \left(3, 0, 0, 1 \right) \right),
      \left(-1, \left(3, 0, 1, 0 \right) \right),
      \left(-1/2,   \left(3, 1, 0, 0 \right) \right)
    \right\}.
  \end{eqnarray*}
  
\end{example}

\begin{example}{FCC Cubic Spline:} \label{ex:fcc}
Finally, as an example of a higher dimensional box spline, we enumerate the sets for the FCC Cubic Spline of Kim et al \cite{fccbox}.
\begin{equation}
\Xi = \begin{bmatrix}
  \ 1 & \ 1 & 1 & -1  & 0 & \ 0\\
  \ 1 & -1 &  0 & \ 0 & 1 & \ 1 \\
  \ 0 & \ 0 & 1 &  \ 1 & 1 & -1 \\
\end{bmatrix}
\end{equation}

\begin{figure}[H]
    \centering
    \includegraphics[scale=0.32]{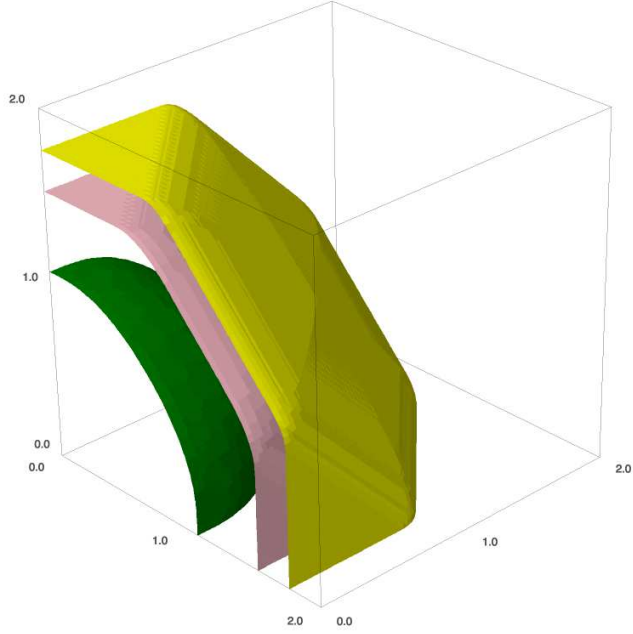}	
    \caption{Iso-contours for the level sets 1/16, 1/256, 1/1024, pictured in green, pink and yellow respectively.}
  \end{figure} \label{fig:fcc}
\end{example}

\begin{eqnarray*}
  S_\Xi  & = & \{
\left( 1, \left(2,1,1\right) \right),  
\left( 1, \left(1,3,2\right) \right),  
\left( -1, \left(0,3,1\right) \right), \\ & &  
\left( 1, \left(-1,3,2\right) \right),  
\left( 1, \left(0,2,0\right) \right),  
\left( -1, \left(2,2,2\right) \right), \\ & &  
\left( 1, \left(0,0,4\right) \right),  
\left( 1, \left(0,2,4\right) \right),  
\left( -1, \left(0,3,3\right) \right), \\ & &  
\left( -1, \left(-2,2,2\right) \right),  
\left( -1, \left(-2,0,2\right) \right),  
\left( -1, \left(2,0,2\right) \right), \\ & &  
\left( -1, \left(0,-1,1\right) \right),  
\left( 1, \left(0,0,0\right) \right),  
\left( 1, \left(-2,1,3\right) \right), \\ & &  
\left( -1, \left(0,-1,3\right) \right),  
\left( 1, \left(2,1,3\right) \right),  
\left( -1, \left(-1,1,0\right) \right), \\ & &  
\left( 1, \left(-1,-1,2\right) \right),  
\left( 1, \left(-2,1,1\right) \right),  
\left( 1, \left(1,-1,2\right) \right), \\ & &  
\left( -1, \left(1,1,0\right) \right),  
\left( -1, \left(-1,1,4\right) \right),  
\left( -1, \left(1,1,4\right) \right) \ \} \\
  P_3 & = &
    \{\left( 1, \left(4,1,1,0,0,0\right) \right),  
\left( 1, \left(4,1,0,0,1,0\right) \right),  
\left( -1, \left(4,0,1,1,0,0\right) \right), \\ & &  
\left( 1, \left(3,2,1,0,0,0\right) \right),  
\left( -1, \left(3,2,0,0,1,0\right) \right),  
\left( 1, \left(4,0,0,0,1,1\right) \right), \\ & &  
\left( -1, \left(4,1,0,0,0,1\right) \right),  
\left( -1, \left(4,1,0,1,0,0\right) \right),  
\left( 1, \left(3,2,0,0,0,1\right) \right), \\ & &  
\left( -1, \left(3,2,0,1,0,0\right) \right) \
  \}.
\end{eqnarray*}
As this box spline has many regions, we omit the enumeration of polynomial pieces. Figure (\ref{fig:fcc}) shows the numerical stability of this form by plotting several iso-contours of this box spline.

\section{Conclusion}
In this work, we provided an explicit decomposition scheme for the piecewise polynomial form of any non-degenerate box spline. 
While there do exist other evaluation schemes for evaluating box splines, they all typically depend on recursively evaluating a box spline, or they are ad-hoc in their construction procedure. 
Our method is general and works for all non-degenerate box-splines. 
However, we only provide a characterization of the splines themselves, when used in an approximation space, it is possible to derive more efficient evaluation schemes for the semi-discrete convolution sum between a box spline and  lattice data. This is an avenue for future work.

We also provide, as a SageMath \cite{sage,Horacsek2016} worksheet, the code to derive the piecewise polynomial form of any box-spline, as well as the code to generate BSP trees that represent the evaluation schemes, and code to generate C++ code from those BSPs.


\bibliographystyle{spmpsci}      


\bibliography{paper}
\newpage
\begin{table}[b]
	\centering
	  \begin{tabular}{| c | c | }
	    \hline
	    Region & Polynomial Piece \\
	    \hline \hline
	    $R_0$  & $  x - y + 1  $ \\
	    $R_1$  & $  -y + 2  $ \\
	    $R_2$  & $  -x + 2  $ \\
	    $R_3$  & $  x  $ \\
	    $R_4$  & $  y  $ \\
	    $R_5$  & $  -x + y + 1  $ \\
	    \hline
	  \end{tabular}
	  \caption{Polynomial regions of the Courant element.}
\end{table}

\begin{table}[b]
	\centering
    \begin{tabular}{| c | c |}
      \hline
      Region & Polynomial Piece \\
      \hline \hline
      $R_{0}, R_{1} $  & $  \frac{1}{4} \, x^{2} - \frac{1}{2} \, x y + \frac{1}{4} \, y^{2} + \frac{3}{2} \, x - \frac{3}{2} \, y + \frac{9}{4}  $ \\

      $R_{2}$  & $  -\frac{1}{4} \, x^{2} - \frac{1}{2} \, x y + \frac{1}{4} \, y^{2} + \frac{3}{2} \, x - \frac{3}{2} \, y + \frac{9}{4}  $ \\
      $R_{3}$  & $  \frac{1}{2} \, y^{2} - 3 \, y + \frac{9}{2}  $ \\
      $R_{4}$  & $  -\frac{1}{2} \, x^{2} + \frac{1}{2} \, x - \frac{1}{2} \, y + \frac{5}{4}  $ \\
      $R_{5}$  & $  -\frac{1}{4} \, x^{2} + \frac{1}{2} \, x y + \frac{1}{4} \, y^{2} - x - 2 \, y + \frac{7}{2}  $ \\
      $R_{6}, R_{7}$  & $  \frac{1}{4} \, x^{2} + \frac{1}{2} \, x y + \frac{1}{4} \, y^{2} - 2 \, x - 2 \, y + 4  $ \\
      $R_{8}$  & $  \frac{1}{2} \, x^{2} + x + \frac{1}{2}  $ \\
      $R_{9}$  & $  \frac{1}{4} \, x^{2} - \frac{1}{2} \, x y - \frac{1}{4} \, y^{2} + \frac{3}{2} \, x + \frac{1}{2} \, y + \frac{1}{4}  $ \\
      $R_{10}$  & $  \frac{1}{4} \, x^{2} + \frac{1}{2} \, x y - \frac{1}{4} \, y^{2} + y - \frac{1}{2}  $ \\
      $R_{11}$  & $  -\frac{1}{2} \, y^{2} + \frac{1}{2} \, x + \frac{3}{2} \, y - \frac{3}{4}  $ \\
      $R_{12}, R_{13}, R_{14}, R_{15}$  & $  -\frac{1}{2} \, x^{2} - \frac{1}{2} \, y^{2} + \frac{1}{2} \, x + \frac{3}{2} \, y - \frac{3}{4}  $ \\
      $R_{16}$  & $  -\frac{1}{2} \, y^{2} - \frac{1}{2} \, x + \frac{3}{2} \, y - \frac{1}{4}  $ \\
      $R_{17}$  & $  \frac{1}{4} \, x^{2} + \frac{1}{2} \, x y - \frac{1}{4} \, y^{2} - 2 \, x + 2  $ \\
      $R_{18}$  & $  \frac{1}{4} \, x^{2} - \frac{1}{2} \, x y - \frac{1}{4} \, y^{2} - \frac{1}{2} \, x + \frac{3}{2} \, y - \frac{1}{4}  $ \\
      $R_{19}$  & $  \frac{1}{2} \, x^{2} - 2 \, x + 2  $ \\
      $R_{20}, R_{21}$  & $  \frac{1}{4} \, x^{2} + \frac{1}{2} \, x y + \frac{1}{4} \, y^{2}  $ \\
      $R_{22}$  & $  -\frac{1}{4} \, x^{2} + \frac{1}{2} \, x y + \frac{1}{4} \, y^{2}  $ \\
      $R_{23}$  & $  -\frac{1}{2} \, x^{2} + \frac{1}{2} \, x + \frac{1}{2} \, y - \frac{1}{4}  $ \\
      $R_{24}$  & $  \frac{1}{2} \, y^{2}  $ \\
      $R_{25}$  & $  -\frac{1}{4} \, x^{2} - \frac{1}{2} \, x y + \frac{1}{4} \, y^{2} + \frac{1}{2} \, x + \frac{1}{2} \, y - \frac{1}{4}  $ \\
      $R_{26}, R_{27}$  & $  \frac{1}{4} \, x^{2} - \frac{1}{2} \, x y + \frac{1}{4} \, y^{2} - \frac{1}{2} \, x + \frac{1}{2} \, y + \frac{1}{4}  $ \\
      \hline
    \end{tabular}
	\caption{Polynomial regions of the ZP-element.}
\end{table}

\begin{table}[!b]
	\centering
  \begin{tabular}{| c | c |}
  \hline
    Region & Polynomial Piece \\
    \hline \hline
      $R_{0}$  & $  \frac{1}{2} \, x^{2} - x y + \frac{1}{2} \, y^{2} + 2 \, x - 2 \, y + 2  $ \\
      $R_{1}$  & $  -\frac{1}{2} \, x^{2} + \frac{1}{4} \, y^{2} + 2 \, x - 2 \, y + 2  $ \\
      $R_{2}, R_{3}$  & $  \frac{1}{4} \, y^{2} - 2 \, y + 4  $ \\
      $R_{4}$  & $  -\frac{1}{2} \, x^{2} + x y - \frac{1}{4} \, y^{2} - x - y + \frac{7}{2}  $ \\
      $R_{5}$  & $  \frac{1}{2} \, x^{2} - 3 \, x + \frac{9}{2}  $ \\
      $R_{6}$  & $  x^{2} - x y + \frac{1}{4} \, y^{2} + x - \frac{1}{2} \, y + \frac{1}{4}  $ \\
      $R_{7}$  & $  \frac{1}{2} \, x^{2} - x y + \frac{1}{4} \, y^{2} + 2 \, x - \frac{1}{2} \, y - \frac{1}{4}  $ \\
      $R_{8}$  & $  -\frac{1}{2} \, x^{2} + 2 \, x - \frac{1}{2} \, y - \frac{1}{4}  $ \\
      $R_{9}, R_{10}$  & $  -x^{2} + x y - \frac{1}{2} \, y^{2} + x + \frac{1}{2} \, y - \frac{3}{4}  $ \\
      $R_{11}$  & $  -\frac{1}{2} \, x^{2} + x y - \frac{1}{2} \, y^{2} - x + \frac{1}{2} \, y + \frac{5}{4}  $ \\
      $R_{12}$  & $  \frac{1}{2} \, x^{2} - \frac{1}{4} \, y^{2} - 3 \, x + \frac{3}{2} \, y + \frac{9}{4}  $ \\
      $R_{13}$  & $  x^{2} - x y + \frac{1}{4} \, y^{2} - 3 \, x + \frac{3}{2} \, y + \frac{9}{4}  $ \\
      $R_{14}$  & $  x^{2} - x y + \frac{1}{4} \, y^{2} + x - \frac{1}{2} \, y + \frac{1}{4}  $ \\
      $R_{15}$  & $  \frac{1}{2} \, x^{2} - \frac{1}{4} \, y^{2} + \frac{1}{2} \, y - \frac{1}{4}  $ \\
      $R_{16}$  & $  -\frac{1}{2} \, x^{2} + x y - \frac{1}{2} \, y^{2} + \frac{1}{2} \, y - \frac{1}{4}  $ \\
      $R_{17}, R_{18}$  & $  -x^{2} + x y - \frac{1}{2} \, y^{2} + x + \frac{1}{2} \, y - \frac{3}{4}  $ \\
      $R_{19}$  & $  -\frac{1}{2} \, x^{2} + x + \frac{1}{2} \, y - \frac{3}{4}  $ \\
      $R_{20}$  & $  \frac{1}{2} \, x^{2} - x y + \frac{1}{4} \, y^{2} - x + \frac{3}{2} \, y + \frac{1}{4}  $ \\
      $R_{21}$  & $  x^{2} - x y + \frac{1}{4} \, y^{2} - 3 \, x + \frac{3}{2} \, y + \frac{9}{4}  $ \\
      $R_{22}$  & $  \frac{1}{2} \, x^{2}  $ \\
      $R_{23}$  & $  -\frac{1}{2} \, x^{2} + x y - \frac{1}{4} \, y^{2}  $ \\
      $R_{24}, R_{25}$  & $  \frac{1}{4} \, y^{2}  $ \\
      $R_{26}$  & $  -\frac{1}{2} \, x^{2} + \frac{1}{4} \, y^{2} + x - \frac{1}{2}  $ \\
      $R_{27}$  & $  \frac{1}{2} \, x^{2} - x y + \frac{1}{2} \, y^{2} - x + y + \frac{1}{2}  $ \\
      \hline
  \end{tabular}
	  \caption{Polynomial regions of the Skewed element.}
\end{table}
\end{document}